\documentclass[
final
, nomarks
]{dmtcs-episciences}

\usepackage[utf8]{inputenc}
\usepackage{subfigure}

\usepackage[round]{natbib}

\usepackage{amssymb}
\usepackage{amsmath}
\usepackage{amsthm}
\usepackage{mathtools}

\usepackage{a4wide}
\usepackage{pgf,tikz}                             

\usepackage{makecell}                             
\usepackage{relsize}

\newcommand\Tstrut{\rule{0pt}{2.6ex}}             
\newcommand\Bstrut{\rule[-0.9ex]{0pt}{0pt}}  

\newcommand{\cL}{\mathcal{L}}                     
\newcommand{\cH}{\mathcal{H}}                     

\newtheorem{lemma}{Lemma}
\newtheorem{proposition}[lemma]{Proposition}
\newtheorem{theorem}[lemma]{Theorem}

\newtheorem{question}{Question}

\theoremstyle{remark}
\newtheorem{remark}{Remark}

\allowdisplaybreaks

\author{Audace A. V. Dossou-Olory\thanks{Supported in part by Stellenbosch University in association with the African Institute for Mathematical Sciences (South Africa).} \and Stephan Wagner\thanks{Supported by the National Research Foundation of South Africa, grant 96236.}}

\title{On the inducibility of small trees}

\affiliation{
  Department of Mathematical Sciences, Stellenbosch University, South Africa
}

\received{2019-4-15}
\revised{2019-8-26}
\accepted{2019-9-10}

\begin{document}

\publicationdetails{21}{2019}{4}{14}{5381}

\maketitle

\begin{abstract}
The quantity that captures the asymptotic value of the maximum number of appearances of a given topological tree (a rooted tree with no vertices of outdegree $1$) $S$ with $k$ leaves in an arbitrary tree with sufficiently large number of leaves is called the inducibility of $S$. Its precise value is known only for some specific families of trees, most of them exhibiting a symmetrical configuration. In an attempt to answer a recent question posed by Czabarka, Sz{\'e}kely, and the second author of this article, we provide bounds for the inducibility $J(A_5)$ of the $5$-leaf binary tree $A_5$ whose branches are a single leaf and the complete binary tree of height $2$. It was indicated before that $J(A_5)$ appears to be `close' to $1/4$. We can make this precise by showing that \begin{math}0.24707\ldots \leq J(A_5) \leq 0.24745\ldots\end{math}. Furthermore, we also consider the problem of determining the inducibility of the tree $Q_4$, which is the only tree among $4$-leaf topological trees for which the inducibility is unknown.
\end{abstract}

\section{Introduction and previous results}

The study of graph inducibility was brought forward in 1975 by Pippenger and Golumbic, who investigated the maximum frequency of $k$-vertex simple graphs occurring as subgraphs within a graph whose number of vertices approaches infinity -- see~\cite{pippenger1975inducibility} for details and first results on the inducibility of graphs. To this day, there is substantial activity regarding this concept. In analogy to~\citep{pippenger1975inducibility}, the inducibility of a rooted tree $S$ with $k$ leaves is defined as the maximum frequency at which $S$ can appear as a subtree induced by $k$ leaves of an arbitrary rooted tree whose number of leaves tends to infinity~\citep{czabarka2016inducibility,AudaceStephanPaper1, DossouOloryWagnerPaper3}. \citet{bubeck2016local} defined the inducibility of a tree $S$ with $k$ vertices as the maximum proportion of $S$ as a subtree among all $k$-vertex subtrees of a tree whose number of vertices tends to infinity. We also mention that \citet{sperfeld2011inducibility} extended the concept of inducibility to monodirected graphs, and also gave bounds (using Razborov's flag algebra method) for some graphs with at most four vertices.

For any of the aforementioned notions of inducibility, can the exact inducibility of trees (graphs) with a moderate size always be determined explicitly? The answer to this question turns out to be either undecidable or negative in general in the original context of simple graphs~\citep{exoo1986dense,sperfeld2011inducibility,hirst2014inducibility,even2015note,bubeck2016local}. The concept of inducibility of a tree with $k$ leaves is still new and the precise value of the inducibility is known only for a few classes of trees, most of them exhibiting a symmetrical configuration. The recent paper~\citep{czabarka2016inducibility} raised some questions on the inducibility of binary trees, one of which is discussed and approximately solved within this note. The present paper also covers a related problem concerning the inducibility of a ternary tree with four leaves.

\medskip
Since the inducibility of trees is a quantity that was only introduced recently, let us first turn to a preliminary account on the subject.

A rooted tree without vertices of outdegree $1$ will be called a \emph{topological} tree as in~\citep{bergeron1998combinatorial,allman2004mathematical,DossouOloryWagnerPaper3}. We are concerned with topological trees with a given number of leaves. If, in addition, every vertex has $d~(\geq 2)$ or fewer children, then the tree will be called a \emph{$d$-ary} tree as in~\citep{AudaceStephanPaper1}. Instead of $2$-ary tree and $3$-ary tree, we shall simply say \emph{binary} tree and \emph{ternary} tree, respectively.

A \emph{leaf-induced subtree} of a topological tree $T$ is any subtree produced in the following three steps: consider a subset $L$ of leaves of $T$; take the minimal subtree containing all the leaves in $L$; suppress all vertices whose outdegree is $1$.

An illustration of this process of finding a leaf-induced subtree of $T$ is shown in Figure~\ref{leaf-induced}. For a topological tree $T$, we shall denote its number of leaves by $|T|$. 

\begin{figure}[!h]\centering  
	\begin{tikzpicture}[thick]
	\node [circle,draw] (r) at (0,0) {};
	
	\draw (r) -- (-2,-2);
	\draw (r) -- (0,-4);
	\draw (r) -- (2,-2);
	\draw (-2,-2) -- (-2.5,-4);
	\draw (-2,-2) -- (-1.5,-4);
	\draw (2,-2) -- (1,-4);
	\draw (2,-2) -- (2,-4);
	\draw (2,-2) -- (3,-4);
	\draw (1.25,-3.5) -- (1.5,-4);
	
	\node [fill,circle, inner sep = 2pt ] at (-2.5,-4) {};
	\node [fill,circle, inner sep = 2pt ] at (-1.5,-4) {};
	\node [fill,circle, inner sep = 2pt ] at (0,-4) {};
	\node [fill,circle, inner sep = 2pt ] at (1,-4) {};
	\node [fill,circle, inner sep = 2pt ] at (1.5,-4) {};
	\node [fill,circle, inner sep = 2pt ] at (2,-4) {};
	\node [fill,circle, inner sep = 2pt ] at (3,-4) {};
	
	\node at (-1.5,-4.5) {$\ell_1$};
	\node at (0,-4.5) {$\ell_2$};
	\node at (1,-4.5) {$\ell_3$};
	\node at (2,-4.5) {$\ell_4$};
	
	\node [circle,draw] (r1) at (6,-2) {};
	
	\draw (r1) -- (5,-4);
	\draw (r1) -- (6,-4);
	\draw (r1) -- (7,-4);
	\draw (6.75,-3.5)--(6.5,-4);
	
	\node [fill,circle, inner sep = 2pt ] at (5,-4) {};
	\node [fill,circle, inner sep = 2pt ] at (6,-4) {};
	\node [fill,circle, inner sep = 2pt ] at (6.5,-4) {};
	\node [fill,circle, inner sep = 2pt ] at (7,-4) {};
	
	\node at (5,-4.5) {$\ell_1$};
	\node at (6,-4.5) {$\ell_2$};
	\node at (6.5,-4.5) {$\ell_3$};
	\node at (7,-4.5) {$\ell_4$};
	
	\end{tikzpicture}
	\caption{A ternary tree $T$ (left) and the subtree induced by the set of leaves $\{\ell_1,\ell_2,\ell_3,\ell_4\}$ of $T$ (right).}\label{leaf-induced}
\end{figure}
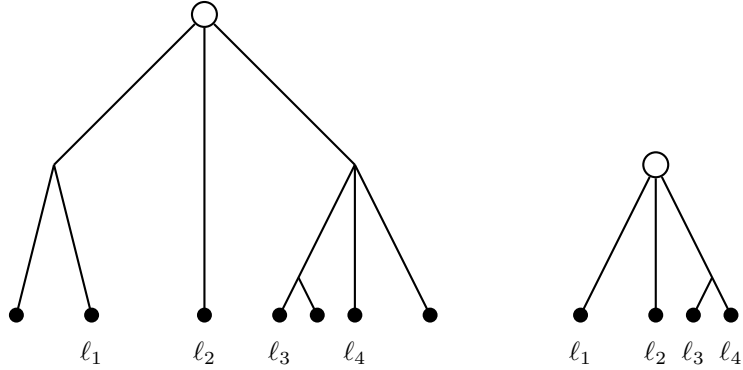
	
By \emph{density} of a topological tree $S$ in $T$, we mean the proportion of all subsets of $|S|$ leaves of $T$ that induce a leaf-induced subtree isomorphic (in the sense of rooted tree isomorphism) to $S$. We shall denote this density by $\gamma(S,T)$. Thus, it makes sense to set $\gamma(S,T)=0$ for $|S|>|T|$.

The \emph{inducibility} of $S$ (as defined and studied in~\citep{DossouOloryWagnerPaper3}) is its maximum density as a leaf-induced subtree of $T$ as the size of $T$ tends to infinity:
\begin{align*}
J(S):=\limsup_{\substack{|T| \to \infty \\T~\text{topological tree}}} \gamma(S,T)=\lim_{n\to \infty} \max_{\substack{|T|=n \\T~\text{topological tree}}} \gamma(S,T)\,.
\end{align*}
The limit is known to exist; see~\citet[Theorem 3]{DossouOloryWagnerPaper3}.

Similarly, when the underlying set over which the supremum is taken is restricted to $d$-ary trees, we define
\begin{align*}
I_d(D):=\limsup_{\substack{|T| \to \infty \\ T~\text{$d$-ary tree}}} \gamma(D,T)=\lim_{n\to \infty} \max_{\substack{|T|=n \\T~\text{$d$-ary tree}}} \gamma(D,T)
\end{align*}
to be the inducibility of a $d$-ary tree $D$ in $d$-ary trees (again, the limit is known to exist---\citet[Theorem 3]{AudaceStephanPaper1}). The subscript $d$ is used to emphasize the fact that we are taking the maximum over the set of all $d$-ary trees.

\medskip
The initial motivation for studying the quantities $J(S)$ and $I_d(D)$ was twofold: first, they are natural tree analogues of the notion of inducibility in graphs (as outlined above), which is classical in graph theory. The other was a concrete application \citep{czabarka2016inducibility} to structures called \emph{tanglegrams}, which consist of two binary trees entangled by a perfect matching between the leaves. In order to estimate the crossing number of random tanglegrams, it was necessary to find (asymptotic) bounds on the number $\gamma(S,T)$.

\medskip
While in the past many results on the inducibility were obtained for graphs, this is not yet the case for trees and many challenging questions remain. The problem of computing the inducibility of a tree appears to be quite difficult even for trees with a small number of leaves---already the inducibilities of some trees with only four or five leaves are not known. The only cases for which an explicit expression is presently known are caterpillars and trees that are highly balanced, thus close to complete $d$-ary trees; cf. papers \citet{czabarka2016inducibility,AudaceStephanPaper1,DossouOloryWagnerPaper2,DossouOloryWagnerPaper3}. The reason why they are manageable is that the trees $T$ for which the maximum
\begin{align*}
\max_{\substack{|T|=n \\T~\text{$d$-ary tree}}} \gamma(D,T)
\end{align*}
is attained have a simple structure (caterpillars and essentially complete $d$-ary trees) in these cases. In general, this structure appears to be much harder to determine, which is why we have to settle for upper and lower bounds in this paper.

Among $5$-leaf binary trees, the tree $A_5$ (see Figure~\ref{treeA5}) is the only one for which the inducibility has not been determined yet. Also, the inducibility of the $4$-leaf ternary tree $Q_4$ shown in Figure~\ref{treeQ4} is unknown. Thus, these are the smallest cases for which we do not have explicit expressions.

\begin{figure}[htbp]\centering  
	\subfigure[The binary tree $A_5$.\label{treeA5}]{\begin{tikzpicture}[thick,level distance=9mm]
		\tikzstyle{level 1}=[sibling distance=19mm]
		\tikzstyle{level 2}=[sibling distance=16mm]
		\tikzstyle{level 3}=[sibling distance=8mm]
		\node [circle,draw]{}
		child {[fill] circle (2pt)}
		child {child {child {[fill] circle (2pt)}child {[fill] circle (2pt)}}child {child {[fill] circle (2pt)}child {[fill] circle (2pt)}}};
		\end{tikzpicture}}
	\hfil
 \subfigure[The ternary tree $Q_4$.\label{treeQ4}]{\begin{tikzpicture}[thick,level distance=12mm]
 	\tikzstyle{level 1}=[sibling distance=20mm]
 	\tikzstyle{level 2}=[sibling distance=9mm]
 	\node [circle,draw]{}
 	child {[fill] circle (2pt)}
 	child {child {[fill] circle (2pt)}child {[fill] circle (2pt)}child {[fill] circle (2pt)}};
 	\end{tikzpicture}}	
	\caption{The topological trees $A_5$ and $Q_4$.}\label{A5andP4}
\end{figure}
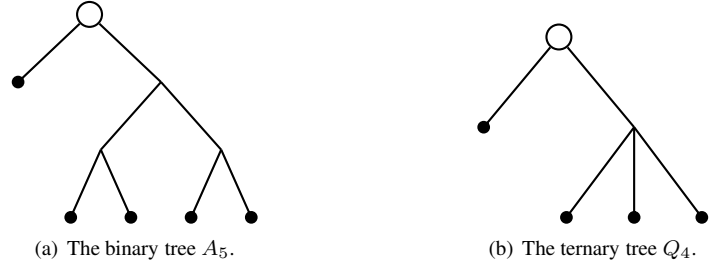

In earlier papers~\citep{AudaceStephanPaper1,DossouOloryWagnerPaper2,DossouOloryWagnerPaper3}, various lower bounds were given on the inducibility of topological trees and thus the inducibilities of $Q_4$ and $A_5$. In this note, we shall propose constructions that yield improved lower bounds on the inducibility of the two trees $Q_4$ and $A_5$. Moreover, using a computer search, we shall be able to bound both the inducibility of $A_5$ in topological trees and the inducibility of $Q_4$ in ternary trees from above.

\medskip
The inducibility of some families of topological trees is known precisely. As such, we have stars, binary caterpillars~\citep{AudaceStephanPaper1}, complete $d$-ary trees and more generally, the so-called even $d$-ary trees~\citep{DossouOloryWagnerPaper2}. We already know the inducibility of all topological trees with at most three leaves: each of them has inducibility $1$, except for the star with three leaves, which has inducibility $(d-2)/(d+1)$ in $d$-ary trees. There are only five different topological trees with four leaves (see Figure~\ref{fourleaves}), and the precise inducibility of four of them is at least partially known:
\begin{align*}
J\big(CD^2_2\big)&=I_d\big(CD^2_2\big)=\frac{3}{7}~~\text{for all}~d~~\text{\citep{czabarka2016inducibility,DossouOloryWagnerPaper2}},\\
J\big(F^2_4\big)&=I_d\big(F^2_4\big)=1~~\text{for all}~d~~\text{\citep{czabarka2016inducibility,AudaceStephanPaper1}}\,,\\
J(S_4)&=1~~\text{\citep{AudaceStephanPaper1}}\,,\\
I_d(S_4)&=\frac{(d-2)(d-3)}{d^2+d+1}~~\text{for all}~d~~\text{\citep{AudaceStephanPaper1}}\,,\\
I_3\big(E^3_4\big)&=\frac{6}{13}~~\text{\citep{DossouOloryWagnerPaper2}}\,,\\
I_d\big(E^3_4\big)&=\text{unknown for}~d>3\,.
\end{align*}

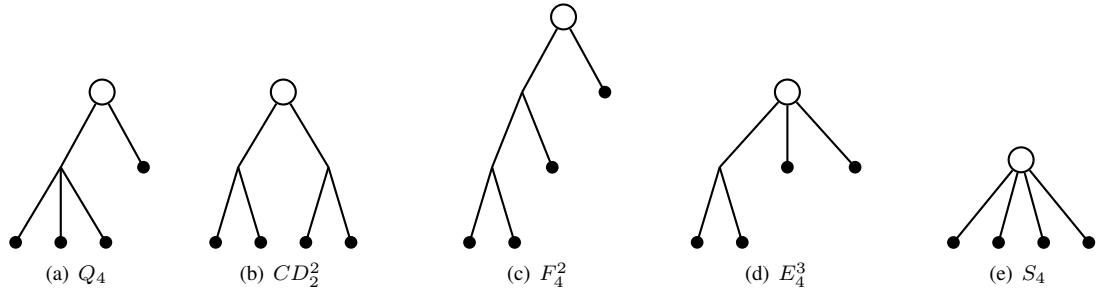
\begin{figure}[htbp] \centering
	\subfigure[$Q_4$]{	\begin{tikzpicture}[thick,level distance=10mm]
		\tikzstyle{level 1}=[sibling distance=11mm]
		\tikzstyle{level 2}=[sibling distance=6mm]
		\node [circle,draw]{}
		child {child {[fill] circle (2pt)}child {[fill] circle (2pt)}child {[fill] circle (2pt)}}
		child {[fill] circle (2pt)};
		\end{tikzpicture}}\qquad
	\subfigure[$CD^2_2$]{	\begin{tikzpicture}[thick,level distance=10mm]
		\tikzstyle{level 1}=[sibling distance=12mm]
		\tikzstyle{level 2}=[sibling distance=6mm]
		\node [circle,draw]{}
		child {child {[fill] circle (2pt)} child {[fill] circle (2pt)}}
		child {child {[fill] circle (2pt)} child {[fill] circle (2pt)}};
		\end{tikzpicture}}\qquad \qquad
	\subfigure[$F^2_4$]{\begin{tikzpicture}[thick,level distance=10mm]
		\tikzstyle{level 1}=[sibling distance=11mm]
		\tikzstyle{level 2}=[sibling distance=8mm]
		\tikzstyle{level 3}=[sibling distance=6mm]
		\node [circle,draw]{}
		child {child {child {[fill] circle (2pt)} child {[fill] circle (2pt)}} child {[fill] circle (2pt)}}
		child {[fill] circle (2pt)};
		\end{tikzpicture}}\qquad \quad
	\subfigure[$E^3_4$]{\begin{tikzpicture}[thick,level distance=10mm]
		\tikzstyle{level 1}=[sibling distance=9mm]
		\tikzstyle{level 2}=[sibling distance=6mm]
		\node [circle,draw]{}
		child {child {[fill] circle (2pt)}child {[fill] circle (2pt)}}
		child {[fill] circle (2pt)}
		child {[fill] circle (2pt)};
		\end{tikzpicture}}\qquad \quad
	\subfigure[$S_4$]{	\begin{tikzpicture}[thick,level distance=11mm]
	\tikzstyle{level 1}=[sibling distance=6mm]
	\node [circle,draw]{}
	child {[fill] circle (2pt)}
	child {[fill] circle (2pt)}
	child {[fill] circle (2pt)}
	child {[fill] circle (2pt)};
	\end{tikzpicture}}
	\caption{All the topological trees with four leaves.}\label{fourleaves}
\end{figure}

When considering binary trees, we notice that there are only three isomorphism types of $5$-leaf trees -- see Figure~\ref{Thetreeswithfiveleaves} -- and the inducibility of two of them has been determined:
\begin{align*}
&J\big(E^2_5\big)=I_d\big(E^2_5\big)=\frac{2}{3}~~
\text{\citep{czabarka2016inducibility,DossouOloryWagnerPaper2}}\,,\\
&J\big(F^2_5\big)=I_d\big(F^2_5\big)=1~~
\text{\citep{czabarka2016inducibility,AudaceStephanPaper1}}
\end{align*}
for all $d$. The inducibility of the binary tree $A_5$ is of particular interest to us, since it is the smallest binary tree for which the inducibility is not known explicitly. In~\citet{czabarka2016inducibility}, the authors considered the problem of computing the inducibility of the tree $A_5$ in binary trees, and mentioned that $I_2(A_5)$ appears to be close to $1/4$. This observation came from a computer experiment, but no explicit sequence of binary trees that would yield a value close to $0.25$ in the limit was given. Here we provide a construction which yields the value $0.24707\ldots$ as a lower bound. We also describe how to perform an efficient computer search and obtain $0.24745\ldots$ as an upper bound on $I_2(A_5)$.

\begin{figure}[!h]\centering
	\subfigure[]{\begin{tikzpicture}[thick,level distance=10mm]
		\tikzstyle{level 1}=[sibling distance=11mm]
		\tikzstyle{level 2}=[sibling distance=5mm]
		\node [circle,draw]{}
		child {child {[fill] circle (2pt)}child {[fill] circle (2pt)}child {[fill] circle (2pt)}child {[fill] circle (2pt)}}
		child {[fill] circle (2pt)};
		\end{tikzpicture}}\qquad \qquad 
	\subfigure[]{\begin{tikzpicture}[thick,level distance=10mm]
		\tikzstyle{level 1}=[sibling distance=14mm]
		\tikzstyle{level 2}=[sibling distance=9mm]
		\tikzstyle{level 3}=[sibling distance=5mm]
		\node [circle,draw]{}
		child {child {child {[fill] circle (2pt)}child {[fill] circle (2pt)}child {[fill] circle (2pt)}}child {[fill] circle (2pt)}}child {[fill] circle (2pt)};
		\end{tikzpicture}}\qquad \qquad 
	\subfigure[]{	\begin{tikzpicture}[thick,level distance=10mm]
		\tikzstyle{level 1}=[sibling distance=15mm]
		\tikzstyle{level 2}=[sibling distance=5mm]
		\node [circle,draw]{}
		child {child {[fill] circle (2pt)}child {[fill] circle (2pt)}}child {child {[fill] circle (2pt)}child {[fill] circle (2pt)}child {[fill] circle (2pt)}};
		\end{tikzpicture}}\qquad \qquad 
  \subfigure[]{\begin{tikzpicture}[thick,level distance=10mm]
  	\tikzstyle{level 1}=[sibling distance=14mm]
  	\tikzstyle{level 2}=[sibling distance=7mm]
  	\tikzstyle{level 3}=[sibling distance=5mm]
  	\node [circle,draw]{}
  	child {child {child {[fill] circle (2pt)}child {[fill] circle (2pt)}}child {[fill] circle (2pt)}child {[fill] circle (2pt)}}
  	child {[fill] circle (2pt)};
  	\end{tikzpicture}}\newline \newline
  \subfigure[$E^2_5$]{\begin{tikzpicture}[thick,level distance=10mm]
   	\tikzstyle{level 1}=[sibling distance=14mm]
   	\tikzstyle{level 2}=[sibling distance=7mm]
   	\tikzstyle{level 3}=[sibling distance=5mm]
   	\node [circle,draw]{}
   	child {child {child {[fill] circle (2pt)}child {[fill] circle (2pt)}}child {[fill] circle (2pt)}}
   	child {child {[fill] circle (2pt)}child {[fill] circle (2pt)}};
   	\end{tikzpicture}}\qquad \qquad 
 \subfigure[$F^2_5$]{\begin{tikzpicture}[thick,level distance=9mm]
 	\tikzstyle{level 1}=[sibling distance=14mm]
 	\tikzstyle{level 2}=[sibling distance=10mm]
 	\tikzstyle{level 3}=[sibling distance=8mm]
 	\tikzstyle{level 4}=[sibling distance=5mm]
 	\node [circle,draw]{}
 	child {child {child {child {[fill] circle (2pt)}child {[fill] circle (2pt)}}child {[fill] circle (2pt)}}child {[fill] circle (2pt)}}
 	child {[fill] circle (2pt)};
 	\end{tikzpicture}}\qquad \qquad 
 \subfigure[$A_5$]{
  	\begin{tikzpicture}[thick,level distance=9mm]
  	\tikzstyle{level 1}=[sibling distance=14mm]
  	\tikzstyle{level 2}=[sibling distance=9mm]
  	\tikzstyle{level 3}=[sibling distance=5mm]
  	\node [circle,draw]{}
  	child {child {child {[fill] circle (2pt)}child {[fill] circle (2pt)}}child {child {[fill] circle (2pt)}child {[fill] circle (2pt)}}}
  	child {[fill] circle (2pt)};
  	\end{tikzpicture}}\qquad \qquad 
  \subfigure[]{\begin{tikzpicture}[thick,level distance=9mm]
   	\tikzstyle{level 1}=[sibling distance=10mm] 
   	\tikzstyle{level 2}=[sibling distance=5mm]
   	\node [circle,draw]{}
   	child {child {[fill] circle (2pt)}child {[fill] circle (2pt)}child {[fill] circle (2pt)}}
   	child {[fill] circle (2pt)}
   	child {[fill] circle (2pt)};
   	\end{tikzpicture}}\newline \newline
    \subfigure[]{\begin{tikzpicture}[thick,level distance=10mm]
    	\tikzstyle{level 1}=[sibling distance=11mm]
    	\tikzstyle{level 2}=[sibling distance=8mm]
    	\tikzstyle{level 3}=[sibling distance=5mm]
    	\node [circle,draw]{}
    	child {child {child {[fill] circle (2pt)}child {[fill] circle (2pt)}}child {[fill] circle (2pt)}}
    	child {[fill] circle (2pt)}
    	child {[fill] circle (2pt)};
    	\end{tikzpicture}}\qquad \quad 
	 \subfigure[]{\begin{tikzpicture}[thick,level distance=9mm]
	 	\tikzstyle{level 1}=[sibling distance=11mm]
	 	\tikzstyle{level 2}=[sibling distance=5mm]
	 	\node [circle,draw]{}
	 	child {child {[fill] circle (2pt)}child {[fill] circle (2pt)}}
	 	child {child {[fill] circle (2pt)}child {[fill] circle (2pt)}}
	 	child {[fill] circle (2pt)};
	 	\end{tikzpicture}}\qquad \quad 
	\subfigure[]{	\begin{tikzpicture}[thick,level distance=10mm]
	  	\tikzstyle{level 1}=[sibling distance=8mm]
	  	\tikzstyle{level 2}=[sibling distance=5mm]
	  	\node [circle,draw]{}
	  	child {child {[fill] circle (2pt)}child {[fill] circle (2pt)}}
	  	child {[fill] circle (2pt)}
	  	child {[fill] circle (2pt)}
	  	child {[fill] circle (2pt)};
	  	\end{tikzpicture}}\qquad \quad 
	  \subfigure[]{\begin{tikzpicture}[thick,level distance=11mm]
	  	\tikzstyle{level 1}=[sibling distance=6mm]
	  	\node [circle,draw]{}
	  	child {[fill] circle (2pt)}
	  	child {[fill] circle (2pt)}
	  	child {[fill] circle (2pt)}
	  	child {[fill] circle (2pt)}
	  	child {[fill] circle (2pt)};
	  	\end{tikzpicture}}
	\caption{All the $5$-leaf topological trees.}\label{Thetreeswithfiveleaves}
\end{figure}
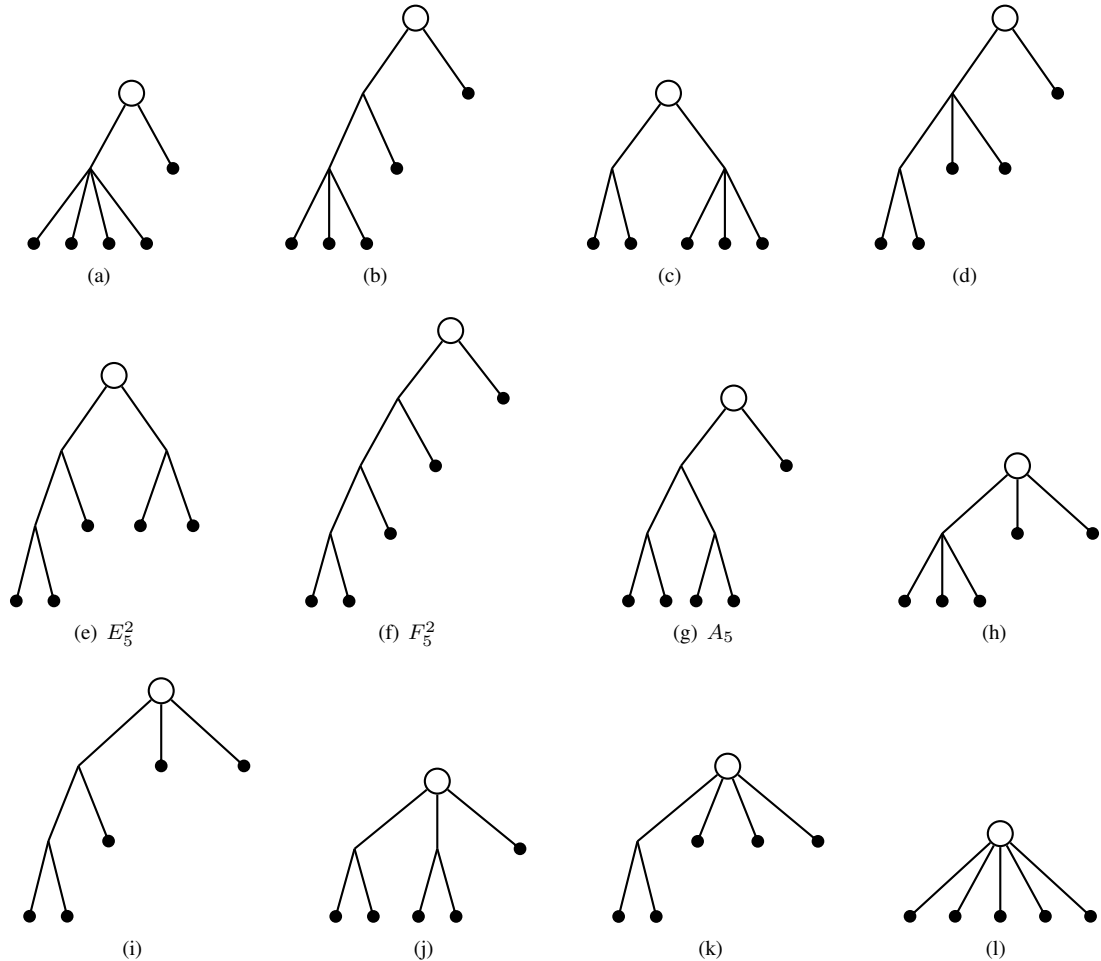

\medskip
In the second part of this work, we consider the problem of finding the inducibility of the ternary tree $Q_4$ in ternary trees. Specifically, we prove that \begin{math}0.1418\ldots \leq I_3(Q_4) \leq 0.1435 \ldots\end{math}. These two trees that we focus on exhibit a non-symmetrical configuration, which makes the computation of their inducibilities harder. For the binary tree $A_5$, we are tempted to conjecture that our candidate is an optimal sequence of binary trees giving the explicit value of $I_2(A_5)$ in the limit, which we obtain as a function of the global maximum of a certain three-variable polynomial over a specific domain.

\section{Statement of results}
Paper~\citep{DossouOloryWagnerPaper3} covers, among other things, the relationship between the degree-restricted inducibility $I_d(S)$ in $d$-ary trees and the general inducibility $J(S)$ in topological trees at large. It was proved in~\citep{DossouOloryWagnerPaper3} that 
\begin{align*}
J(S)= \lim_{d\to \infty} I_d(S)\,.
\end{align*}

A $d$-ary tree will be called a \emph{strictly} $d$-ary tree if each of its internal vertices has exactly $d$ children. By a result in~\citep[Theorem 5]{AudaceStephanPaper1}, we also know that the underlying set over which the maximum density in $d$-ary trees is taken can be reduced to strictly $d$-ary trees, that is
\begin{align*}
I_d(S)=\lim_{n \to \infty} \max_{\substack{|T|=n \\ T~\text{strictly $d$-ary tree}}}\gamma(S,T)\,.
\end{align*}

\medskip
In~\citep{czabarka2016inducibility}, the authors formulated some questions and conjectures on the inducibility in binary trees, one of which was solved recently in~\citep{AudaceStephanPaper1}. Among the questions posed, one of them asks for the inducibility of the $5$-leaf binary tree $A_5$ (see Figure~\ref{A5andP4}). As mentioned in the introduction, this problem appears to be quite hard and finding a sequence of binary trees that yields $I_2(A_5)$ in the limit also appears to be a difficult task. \citet{czabarka2016inducibility} further mentioned that $I_2(A_5)$ is close to $1/4$, which will be made more precise here with the following result:

\begin{theorem}\label{lowerbforA5}
	For the binary tree $A_5$, we have
	\begin{align*}
	0.247071 \leq J(A_5)=I_2(A_5) \leq \frac{32828685715097}{132667832500200} \approx 0.247450\,.
	\end{align*}
\end{theorem}

\medskip
As part of the ingredients needed to prove this result, let us define a new class of binary trees (which is already considered in recent papers~\citep{czabarka2016inducibility,DossouOloryWagnerPaper2}).

\medskip
The \emph{even} binary tree $E^2_n$ with $n$ leaves is obtained recursively as follows:
\begin{itemize}
	\item $E^2_1$ is the tree with only one vertex;
	\item for $n>1$, the branches of $E^2_n$ are the even binary trees $E^2_{\lfloor n/2 \rfloor}$ and $E^2_{\lceil n/2 \rceil}$.
\end{itemize}
An example of an even binary tree can be found in Figure~\ref{E2.7}.
\begin{figure}[htbp]\centering  
	\begin{tikzpicture}[thick,level distance=10mm]
	\tikzstyle{level 1}=[sibling distance=25mm]
	\tikzstyle{level 2}=[sibling distance=11mm]
	\tikzstyle{level 3}=[sibling distance=6.5mm]
	\node [circle,draw]{}
	child {child {[fill] circle (2pt)}child {child {[fill] circle (2pt)}child {[fill] circle (2pt)}}}
	child {child {child {[fill] circle (2pt)}child {[fill] circle (2pt)}}child {child {[fill] circle (2pt)}child {[fill] circle (2pt)}}};
	\end{tikzpicture}
	\caption{The even binary tree $E^2_7$ with seven leaves.}\label{E2.7}
\end{figure}
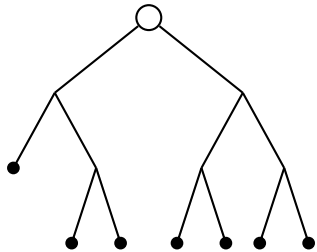

We shall prove the upper bound in Theorem~\ref{lowerbforA5} by means of an algorithmic approach. For the lower bound, we shall make use of the binary tree $S(n_1,n_2,n_3,n_4)$ whose rough picture is shown in Figure~\ref{Gene Tree for A2,5}, where each triangle represents an even binary tree. More specifically, to obtain the tree $S(n_1,n_2,n_3,n_4)$, we take the $4$-leaf binary tree whose internal vertices form a path beginning at the root (the square vertex on top in Figure~\ref{Gene Tree for A2,5}), and identify the four leaves with the even binary trees whose number of leaves is $n_1,n_2,n_3,n_4$, respectively in this order (starting with the top leaf attached to the root). 

\begin{figure}[htbp]\centering
	\begin{tikzpicture}[]
	\draw (3.,13.)-- (5.,11.);
	\draw (5.,11.)-- (7.,9.);
	\draw (7.,9.)-- (9.,7.);
	\draw (9.,7.)-- (10.,6.);
	\draw (3.,13.)-- (2.,12.);
	\draw (5.,11.)-- (4.,10.);
	\draw (7.,9.)-- (6.,8.);
	\draw (9.,7.)-- (8.,6.);
	\draw (8.,6.)-- (10.,6.);
	\draw (5.,7.)-- (7.,7.);
	\draw (6.,8.)-- (5.,7.);
	\draw (6.,8.)-- (7.,7.);
	\draw (4.,10.)-- (3.,9.);
	\draw (4.,10.)-- (5.,9.);
	\draw (3.,9.)-- (5.,9.);
	\draw (2.,12.)-- (1.,11.);
	\draw (2.,12.)-- (3.,11.);
	\draw (1.,11.)-- (3.,11.);
	\draw (1.7245451798215154,10.800158178264521) node[anchor=north west] {$n_1$};
	\draw (3.745707296485689,8.854825855992067) node[anchor=north west] {$n_2$};
	\draw (5.762097589752763,6.8509872984632985) node[anchor=north west] {$n_3$};
	\draw (8.778612588146123,5.768621946221482) node[anchor=north west] {$n_4$};
	\begin{scriptsize}
	\draw [fill=black] (3.,13.) ++(-4.0pt,0 pt) -- ++(4.0pt,4.0pt)--++(4.0pt,-4.0pt)--++(-4.0pt,-4.0pt)--++(-4.0pt,4.0pt);
	\draw [fill=black] (2.,12.) circle (2.0pt);
	\draw [fill=black] (5.,11.) circle (2.5pt);
	\draw [fill=black] (4.,10.) circle (2.0pt);
	\draw [fill=black] (7.,9.) circle (2.5pt);
	\draw [fill=black] (6.,8.) circle (2.0pt);
	\draw [fill=black] (9.,7.) circle (2.5pt);
	\draw [fill=black] (8.,6.) circle (1.5pt);
	\draw [fill=black] (10.,6.) circle (1.5pt);
	\draw [fill=black] (1.,11.) circle (1.5pt);
	\draw [fill=black] (3.,11.) circle (1.5pt);
	\draw [fill=black] (3.,9.) circle (1.5pt);
	\draw [fill=black] (5.,9.) circle (1.5pt);
	\draw [fill=black] (5.,7.) circle (1.5pt);
	\draw [fill=black] (7.,7.) circle (1.5pt);
	\end{scriptsize}
	\end{tikzpicture}
	\caption{The binary tree $S(n_1,n_2,n_3,n_4)$ described for Theorem~\ref{lowerbforA5}.}\label{Gene Tree for A2,5}
\end{figure}
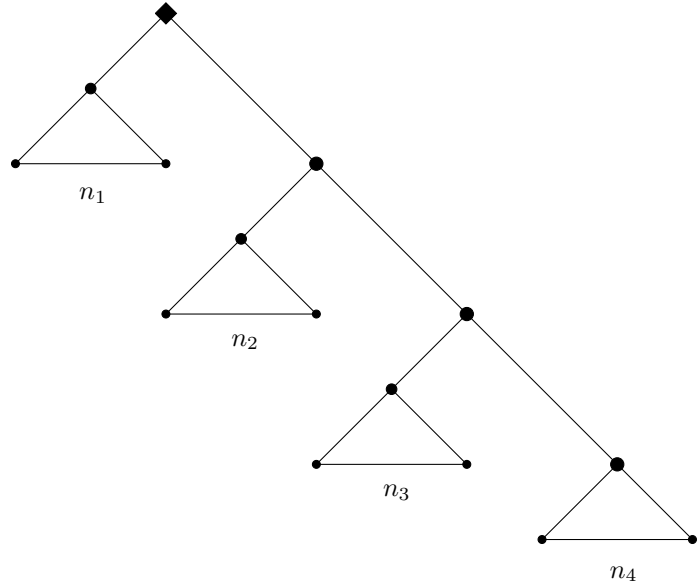

As a next step, we set up a formula for the number of copies of $A_5$ in $S(n_1,n_2,n_3,n_4)$; this formula is used together with a result on even binary trees from~\citet{DossouOloryWagnerPaper2} to derive an asymptotic formula for \begin{math}c(A_5,S(n_1,n_2,n_3,n_4))\end{math} as \begin{math}n=n_1+n_2+n_3+n_4\to \infty\end{math}. Finally, we compute (at least approximately) the global maximum of the main term in the asymptotic formula of the density \begin{math}\gamma(A_5,S(n_1,n_2,n_3,n_4))\end{math} in the region defined by \begin{math}0<n_1,n_2,n_3,n_4<n\end{math} and \begin{math}n_1+n_2+n_3+n_4=n\end{math}.

As a closing comment, when we consider five or more even binary trees instead of four in the tree configuration of Figure~\ref{Gene Tree for A2,5}, we do not seem to get a better lower bound. We therefore expect our construction to be best possible.

\medskip
Among the topological trees with fewer than five leaves, the $4$-leaf ternary tree $Q_4$ (Figure~\ref{A5andP4}) is the only one for which we are yet to determine an exact inducibility. What is the inducibility of $Q_4$ (at least in ternary trees)? In what follows, we shall derive a lower and upper bound on $I_3(Q_4)$. Our second main theorem reads as follows:

\begin{theorem}\label{accurate lower bound for P4}
	For the ternary tree $Q_4$, we have
	\begin{align*}
	0.141827 \approx \frac{59}{416} \leq I_3(Q_4) \leq \frac{73848853}{514606225} \approx 0.143506\,.
	\end{align*}
\end{theorem}

\medskip
The proof of the lower bound in Theorem~\ref{accurate lower bound for P4} is accomplished by an explicit construction (as in  Theorem~\ref{lowerbforA5}), while the upper bound is obtained by means of a computer search. We defer them to Section~\ref{Proof2}.

\medskip
The \emph{star} with $k$ leaves is obtained by joining $k$ distinct vertices to a new vertex (the root of the star). We shall denote it with the symbol $S_k$.

The \emph{complete $d$-ary tree of height $h$} is the strictly $d$-ary tree in which the distance from every leaf to the root is $h$. Such a tree has $d^h$ leaves in total and shall be denoted with the symbol $CD^d_h$.

\medskip
For a positive integer $k\geq 3$, denote by $Q_{k}$ the tree whose branches are $S_{k-1}$ and $S_1$ (the single leaf). The following proposition will serve as an intermediary result to proving a new lower bound on the inducibility of the tree $Q_4$. Its proof will be given in Section~\ref{Proof2}.

\begin{proposition}\label{prop: copy of P4}
	For all positive integers $d,k,$ and $h$ such that $d \geq 2$ and $k \geq 3$, the formula
	\begin{align*}
	c\big(Q_k,CD^d_h \big)=\frac{(d-1) \binom{d}{k-1}}{d^{k-1}-d}\cdot d^h \Bigg(\frac{d^{(k-1) h}-d^{k-1}}{d^{k-1}-1}-\frac{d^h-d}{d-1} \Bigg)
	\end{align*}
	holds. In particular, we have
	\begin{align*}
	I_d(Q_k)\geq \frac{k! (d-1) \binom{d}{k-1}}{(d^{k-1}-d) (d^{k-1}-1)}
	\end{align*}
	for every $d \geq 2$ and every $k\geq 3$.
\end{proposition}

The next proposition shows that the bounds mentioned in Theorems~\ref{lowerbforA5} and \ref{accurate lower bound for P4} are much better than the natural bounds provided by the complete $d$-ary trees, cf.~\citet{DossouOloryWagnerPaper2}.

\begin{proposition}
	For the trees $Q_4$ and $A_5$, we have
	\begin{align*}
	\lim_{h \to \infty}\gamma\big(Q_4,CD^3_h\big)=\frac{1}{13}
	\end{align*}
	and
	\begin{align*}
	\lim_{h \to \infty}\gamma\big(A_5,CD^2_h\big)=\frac{1}{7}\,.
	\end{align*}
\end{proposition}

\begin{proof}
	The specialisation $d=3$ and $k=4$ in Proposition~\ref{prop: copy of P4} yields
	\begin{align*}
	\lim_{h \to \infty}\gamma\big(Q_4,CD^3_h\big)=\frac{1}{13}\,.
	\end{align*}
	
	As a special case of a result in~\citep[Theorem 1]{DossouOloryWagnerPaper2}, we know that
	\begin{align*}
	\lim_{h \to \infty}\gamma\big(A_5,CD^2_h\big)=\frac{2\cdot 5}{2^5-2}\cdot I_2\big(CD^2_2\big)\,,
	\end{align*}
	while it was proved in the same source (see also~\citet[Proposition~2]{czabarka2016inducibility}) that 
	\begin{align*}
	I_2\big(CD^2_2\big)=\frac{3}{7}\,.
	\end{align*}
This completes the proof of the proposition.
\end{proof}

\section{An algorithm for the maximum}\label{sec:ALGO}

Our next theorem will be used to prove the upper bound on the inducibility of each of the trees $A_5$ and $Q_4$. Here, we shall only discuss the tree $A_5$ (the case of $Q_4$ is analogous, as will become clear from the proof). We know from \citet[Theorem~3]{AudaceStephanPaper1} that
\begin{align*}
I_d(S) \leq \max_{\substack{|T| = n \\ T \text{ $d$-ary tree}}} \gamma(S,T)
\end{align*}
for all $d$-ary trees $S$ and $n \geq |S|$. Thus it suffices to determine the value on the right (which can be shown to be decreasing in $n$ as $n \geq |S|$) for as large a value of $n$ as possible to obtain an upper bound. This will be the main goal of this section, where an algorithm for this purpose will be presented. We first need a series of lemmas.

\medskip
If $v$ is a vertex of a topological tree $T$, then the subtree $T[v]$ consisting of $v$ and all its descendants in $T$ is called a \emph{fringe} subtree of $T$. In other words, $T[v]$ is the subtree of $T$ rooted at $v$.

\begin{lemma}\label{lem:two_terms}
	Let $v$ be a vertex of a binary tree $T$, and let $T[v]$ be the fringe subtree rooted at $v$. The number of copies of $A_5$ in $T$ can be expressed as
	\begin{align*}
	c(A_5,T) = c(A_5,T[v]) + (|T| - |T[v]|) c(CD_2^2,T[v]) + R\,,
	\end{align*}
	where $R$ only depends on the size of $T[v]$ (and the rest of $T$), but not its precise structure.
\end{lemma}

\begin{proof}
	If a set of leaves contains at most three leaves of $T[v]$, then there is only one possibility for the tree induced by them inside of $T[v]$. Thus the number of copies of $A_5$ in $T$ that contain at most three leaves of $T[v]$ only depends on the size of $T[v]$, but not its shape. This leaves us with 
	\begin{itemize}
		\item copies of $A_5$ that are entirely contained in $T[v]$; their number is clearly $c(A_5,T[v])$,
		\item copies of $A_5$ that contain precisely four leaves of $T[v]$; there are $|T| - |T[v]|$ other leaves, and the four leaves in $T[v]$ have to induce a copy of $CD_2^2$ to obtain a copy of $A_5$. Thus the number of these copies is \begin{math}(|T| - |T[v]|) c(CD_2^2,T[v])\end{math}.
	\end{itemize}
	The statement of the lemma follows.
\end{proof}

\begin{lemma}\label{lem:aux1}
	Let $v$ be a vertex of a binary tree $T$, and let $T[v]$ be the fringe subtree rooted at $v$. Let $S$ be a binary tree of the same size as $T[v]$ that satisfies
	\begin{align*}
c(CD_2^2,S) \geq c(CD_2^2,T[v]) \quad \text{and} \quad c(A_5,S) \geq c(A_5,T[v])\,,
	\end{align*}
	at least one of them with strict inequality. Let $T'$ be obtained from $T$ by replacing $T[v]$ with $S$; then we have
\begin{equation*}
c(A_5,T') > c(A_5,T).
\end{equation*}
\end{lemma}

\begin{proof}
	This is immediate from the previous lemma.
\end{proof}

\begin{lemma}\label{lem:aux2}
	Let $v$ be a vertex of a binary tree $T$, and let $T[v]$ be the fringe subtree rooted at $v$. Let $S_1$ and $S_2$ be two binary trees of the same size as $T[v]$ that satisfy
	\begin{align*}
	c(CD_2^2,S_1) > c(CD_2^2,T[v]) > c(CD_2^2,S_2)
	\end{align*}
	and
	\begin{align*}
 c( A_5,S_1) < c(A_5,T[v]) < c(A_5,S_2)\,.
	\end{align*}
	Suppose further that
	\begin{equation}\label{eq:slopes}
	\frac{c(A_5,S_1) - c(A_5,T[v])}{c(CD_2^2,S_1) - c(CD_2^2,T[v])} \geq \frac{c(A_5,T[v]) - c(A_5,S_2)}{c(CD_2^2,T[v]) - c(CD_2^2,S_2)}.
	\end{equation}
	Let $T_1$ and $T_2$ be obtained from $T$ by replacing $T[v]$ with $S_1$ and $S_2$ respectively; then we have
	\begin{equation}\label{eq:one_of}
	\max \big(c(A_5,T_1),c(A_5,T_2) \big) \geq c(A_5,T).
	\end{equation}
	If strict inequality holds in~\eqref{eq:slopes}, then we also have strict inequality in~\eqref{eq:one_of}.
\end{lemma}

\begin{proof}
	Let $k = |T| - |T[v]|$. By Lemma~\ref{lem:two_terms}, we have
	\begin{align*}
	c(A_5,T_1) - c(A_5,T) &= c(A_5,S_1) - c(A_5,T[v]) + k \big( c(CD_2^2,S_1) - c(CD_2^2,T[v]) \big) \\
	&= \big( c(CD_2^2,S_1) - c(CD_2^2,T[v]) \big) \Big( k + \frac{c(A_5,S_1) - c(A_5,T[v])}{c(CD_2^2,S_1) - c(CD_2^2,T[v])} \Big).
	\end{align*}
	If
\begin{equation*}
\frac{c(A_5,S_1) - c(A_5,T[v])}{c(CD_2^2,S_1) - c(CD_2^2,T[v])} \geq -k,
\end{equation*}
	then we are done, since $c(A_5,T_1) \geq c(A_5,T)$. Otherwise,~\eqref{eq:slopes} implies that
\begin{equation*}
\frac{c(A_5,T[v]) - c(A_5,S_2)}{c(CD_2^2,T[v]) - c(CD_2^2,S_2)} < -k.
\end{equation*}
	Now it follows that
	\begin{align*}
	c(A_5,T_2) - c(A_5,T) &= c(A_5,S_2) - c(A_5,T[v]) + k \big( c(CD_2^2,S_2) - c(CD_2^2,T[v]) \big) \\
	&= \big( c(CD_2^2,S_2) - c(CD_2^2,T[v]) \big) \Big( k + \frac{c(A_5,T[v]) - c(A_5,S_2)}{c(CD_2^2,T[v]) - c(CD_2^2,S_2)} \Big) > 0,
	\end{align*}
	so $c(A_5,T_2) \geq c(A_5,T)$. Either way, we have~\eqref{eq:one_of}. Equality can only hold if both quotients in~\eqref{eq:slopes} are equal to $-k$. This completes the proof.
\end{proof}

\begin{lemma}\label{lem:aux3}
	Let $v$ be a vertex of a binary tree $T$, and let $T[v]$ be the fringe subtree rooted at $v$. Let $S$ be a binary tree of the same size as $T[v]$ that satisfies
	\begin{align*}
	c(CD_2^2,S) > c(CD_2^2,T[v])
	\end{align*}
	and
	\begin{align*}
	c(A_5,S) < c(A_5,T[v])\,.
	\end{align*}
	Suppose further that
	\begin{equation}\label{eq:slope2}
	\frac{c(A_5,S) - c(A_5,T[v])}{c(CD_2^2,S) - c(CD_2^2,T[v])} \geq |T[v]| - |T|.
	\end{equation}
	Let $T'$ be obtained from $T$ by replacing $T[v]$ with $S$; then we have
	\begin{equation}\label{eq:third}
	c(A_5,T') \geq c(A_5,T).
	\end{equation}
	If strict inequality holds in~\eqref{eq:slope2}, then we also have strict inequality in~\eqref{eq:third}.
\end{lemma}

\begin{proof}
	As in the proof of the previous lemma, we have
	\begin{align*}
c(A_5,T') - c(A_5,T) = \big( c(CD_2^2,S) - c(CD_2^2,T[v]) \big) \Big( |T| - |T[v]| + \frac{c(A_5,S) - c(A_5,T[v])}{c(CD_2^2,S) - c(CD_2^2,T[v])} \Big)\,.
	\end{align*}
	The statement follows immediately.
\end{proof}

\medskip
Now we are ready to describe the algorithm to determine the maximum number of copies of $A_5$ in a binary tree with $n$ leaves. To this end, we define a sequence of sets of binary trees: intuitively speaking, $\mathcal{L}(n)$ consists of trees with $n$ leaves that can potentially occur as fringe subtrees of ``optimal'' trees, \textit{i.e.},
binary trees that maximize the number of copies of $A_5$. A formal recursive definition will be provided below. We also associate every tree $T$ with the pair \begin{math}P(T) = (c(A_5,T),c(CD_2^2,T))\end{math}, which can be interpreted as a point in the plane, and we set
\begin{align*}
L(n) = \{P(T) \,:\, T \in \cL(n)\}
\end{align*}
for every $n$. The sets $\cL(n)$ are recursively defined as follows:

\begin{enumerate}
	\item The set $\cL(1)$ only consists of one tree, which only has a single vertex.
	\item For $n > 1$, we consider all binary trees with $n$ leaves for which each branch lies in one of the sets $\cL(m)$ for some $m < n$. Clearly, if one branch lies in $\cL(k)$, the other has to lie in $\cL(n-k)$. For reasons to become clear later (essentially, we are applying Lemma~\ref{lem:aux3}), we will be even more restrictive: we consider all binary trees with $n$ leaves whose branches both lie in
	\begin{multline*}
	\bigcup_{m < n} \Big\{ T \in \cL(m) \,:\, \text{there is no } S \in \cL(m) \text{ such that } c(CD_2^2,S) > c(CD_2^2,T), \\
	c(A_5,S) < c(A_5,T), \text{ and } \frac{c(A_5,S) - c(A_5,T[v])}{c(CD_2^2,S) - c(CD_2^2,T)} \geq m-n \Big\}.
	\end{multline*}
	This gives us a preliminary set $\cH_1(n)$.
	\item If there are two trees $T$ and $T'$ in $\cH_1(n)$ such that
	\begin{align*}
c(CD_2^2,T) \geq c(CD_2^2,T') \qquad \text{and} \qquad c(A_5,T) \geq c(A_5,T')\,,
	\end{align*}
	remove $T'$ from $\cH_1(n)$. If we have equality in both inequalities, we can arbitrarily remove either $T$ or $T'$. In geometric terms, the condition means that the point $P(T')$ lies to the left and below the point $P(T)$ in the plane. We repeat this step until there are no two trees $T$ and $T'$ satisfying the aforementioned condition anymore. At the end, we are left with a set $\cH_2(n)$. 
	\item As a final reduction step, we eliminate all trees $T$ from $\cH_2(n)$ for which there exist two trees $S_1$ and $S_2$ in $\cH_2(n)$ such that
	the inequalities of Lemma~\ref{lem:aux2} hold, \textit{i.e.},
	\begin{align*}
c(CD_2^2,S_1) > c(CD_2^2,T) > c(CD_2^2,S_2)
	\end{align*}
	and
	\begin{align*}
	c(A_5,S_1) < c(A_5,T) < c(A_5,S_2)
	\end{align*}
	as well as
	\begin{align*}
\frac{c(A_5,S_1) - c(A_5,T)}{c(CD_2^2,S_1) - c(CD_2^2,T)} \geq \frac{c(A_5,T) - c(A_5,S_2)}{c(CD_2^2,T) - c(CD_2^2,S_2)}\,.
	\end{align*}
	
	Considering the set of points \begin{math}\{P(T) \,: \, T \in \cH_2(T)\}\end{math} in the plane, this amounts to taking the upper envelope of the points. The resulting set after this reduction is
	$\cL(n)$. At this point, we can arrange the elements of $\cL(n)$ as a list of trees $T_1,T_2,\ldots,T_r$ such that
	\begin{align*}
	 c(CD_2^2,T_1) &< c(CD_2^2,T_2) < \cdots < c(CD_2^2,T_r)\,,\\
	 c(A_5,T_1) &> c(A_5,T_2) > \cdots > c(A_5,T_r)\,,
	\end{align*}
	and the sequence of ``slopes''
	\begin{align*}
	\frac{c(A_5,T_{j+1}) - c(A_5,T_j)}{c(CD_2^2,T_{j+1}) - c(CD_2^2,T_j)}
	\end{align*}
	is strictly decreasing. This also makes it easier to construct the set in step (2): the trees from $\cL(m)$ that are allowed as branches are precisely those starting from the point where the slope is less than $m-n$.
\end{enumerate}

Due to the rules of the two elimination steps, the following holds for all $T \in \cH_1(n)$ at the end:
\begin{itemize}
	\item Either there exists an $S \in \cL(n)$ (possibly $T = S$) such that
	\begin{align*}
	c(CD_2^2,S) \geq c(CD_2^2,T) \qquad \text{and} \qquad c(A_5,S) \geq c(A_5,T)\,,
	\end{align*}
	\item or there exist two trees $S_1,S_2 \in \cL(n)$ such that
	\begin{align*}
	c(CD_2^2,S_1) > c(CD_2^2,T) > c(CD_2^2,S_2),\ c(A_5,S_1) < c(A_5,T) < c(A_5,S_2)
	\end{align*}
	and
	\begin{align*}
	\frac{c(A_5,S_1) - c(A_5,T)}{c(CD_2^2,S_1) - c(CD_2^2,T)} \geq \frac{c(A_5,T) - c(A_5,S_2)}{c(CD_2^2,T) - c(CD_2^2,S_2)}\,.
	\end{align*}

\end{itemize}

\medskip
The following theorem shows that the maximum of $c(A_5,T)$ for binary trees $T$ with a given number of leaves can be determined purely by focusing on the sets $\cL(n)$.

\begin{theorem}\label{Thm:forAlgo}
	For every positive integer $n$, there exists a binary tree $M_n$ with $n$ leaves such that
	\begin{align*}
	c(A_5,M_n) = \max_{\substack{|T| = n\\ T~\text{binary tree}}} c(A_5,T)
	\end{align*}
	and all fringe subtrees of $M_n$ (including $M_n$ itself) lie in $\bigcup_{k \geq 1} \cL(k)$. In particular,
	\begin{align*}
	\max_{\substack{|T| = n\\ T~\text{binary tree}}} c(A_5,T) = \max_{T \in \cL(n)} c(A_5,T)\,.
	\end{align*}
\end{theorem}

\begin{proof}
	Suppose that the statement does not hold, and let $m$ be minimal with the property that there is a positive integer $n$ such that every ``optimal'' tree (tree attaining the maximum \begin{math}\max_{|T| = n} c(A_5,T)\end{math}) has a fringe subtree with $m$ or fewer leaves that does not lie in \begin{math}\bigcup_{1 \leq k \leq m} \cL(k)\end{math}. Clearly, $m > 1$.
	
	By our choice of $m$, there must be an optimal tree $T$ with $n$ leaves for which all fringe subtrees with less than $m$ leaves lie in \begin{math}\bigcup_{1 \leq k < m} \cL(k)\end{math}. Among all possible choices of $T$, we can choose one for which the number of $m$-leaf fringe subtrees that
	do not lie in $\cL(m)$ is minimal. Consider one of these fringe subtrees $T[v]$. Both its branches lie in \begin{math}\bigcup_{1 \leq k < m} \cL(k)\end{math}, which leaves us with the following possible reasons why $T[v]$ is not in $\cL(m)$:
	\begin{itemize}
		\item The branches of $T[v]$ do not satisfy the condition of step (2) in the construction of $\cL(n)$ (\textit{i.e.}, $T[v]$ does not even lie in $\cH_1(m)$). Suppose that for one of the branches $B$, there is a tree $S$ in $\cL_{|B|}$ such that
		\begin{align*}
	  c(CD_2^2,S) > c(CD_2^2,B), c(A_5,S) < c(A_5,B)\,,
		\end{align*}
		and
		\begin{align*}
		\frac{c(A_5,S) - c(A_5,B)}{c(CD_2^2,S) - c(CD_2^2,B)} \geq |B|-|T[v]| \geq |B|-|T|\,.
		\end{align*}
	
		We can replace $B$ by $S$, and do likewise with the other branch of $T[v]$ if necessary. We either reach a contradiction to the optimality of $T[v]$ by means of Lemma~\ref{lem:aux3} immediately, or (if equality holds above) a new tree that is still optimal, but where $T[v]$ has been replaced by a tree in $\cL(m)$, again contradicting the choice of $T$. So for the remaining cases, we can at least assume that $T[v] \in \cH_1(m)$.
		
		\item There is a binary tree $S \in \cL(m)$ such that
		\begin{align*}
	c(CD_2^2,S) \geq c(CD_2^2,T[v]) \qquad \text{and} \qquad c(A_5,S) \geq c(A_5,T[v])\,.
		\end{align*}
	
		In this case, we can replace $T[v]$ by $S$ to obtain a new tree with at least as many copies of $A_5$ as $T$ by Lemma~\ref{lem:aux1}. This contradicts our choice of $T$ (it is either not optimal, or it does not have the smallest number of $m$-leaf
		fringe subtrees that do not lie in $\cL(m)$).
		\item There are binary trees $S_1,S_2 \in \cL(m)$ such that
		\begin{align*}
		c(CD_2^2,S_1) > c(CD_2^2,T) > c(CD_2^2,S_2),\ c(A_5,S_1) < c(A_5,T) < c(A_5,S_2)
		\end{align*}	
		and
		\begin{align*}
		\frac{c(A_5,S_1) - c(A_5,T)}{c(CD_2^2,S_1) - c(CD_2^2,T)} \geq \frac{c(A_5,T) - c(A_5,S_2)}{c(CD_2^2,T) - c(CD_2^2,S_2)}\,.
		\end{align*}
		In this case, we can replace $T[v]$ by either $S_1$ or $S_2$ to obtain a contradiction in the same way as in the previous case (now by means of Lemma~\ref{lem:aux2}).
	\end{itemize}
	Since we reach a contradiction in all possible cases, the proof is complete.
\end{proof}

\medskip
For a practical implementation of this algorithm, it actually suffices to work with the lists
\begin{align*}
L(n) = \{P(T) \,:\, T \in \cL(n)\}
\end{align*}
that contain the values of \begin{math}P(T) = (c(A_5,T),c(CD_2^2,T))\end{math}. These values can be calculated recursively: if the branches of a binary tree $T$ are $B_1$ and $B_2$, we have
\begin{align}\label{A5inT}
c(A_5,T) = c(A_5,B_1) + c(A_5,B_2) + |B_1|c(CD_2^2,B_2) + |B_2|c(CD_2^2,B_1)
\end{align}
and
\begin{align}\label{eqforCD2}
c(CD_2^2,T) = c(CD_2^2,B_1) + c(CD_2^2,B_2) + \binom{|B_1|}{2} \binom{|B_2|}{2}\,.
\end{align}
These formulas can be explained as follows:
\begin{itemize}
	\item A subset of five leaves of the leaf-set of $T$ can either be a subset of leaves of $B_1$, or a subset of leaves of $B_2$, or splits into leaves of both $B_1$ and $B_2$. In the latter case, the split must be of the type $1-4$ (or $4-1$) as the branches of $A_5$ are a single vertex and $CD^2_2$. Moreover, the four leaves that lie in one branch have to induce $CD_2^2$ there. This proves the recursion for $A_5$.
	\item Four leaves of $T$ that induce the tree $CD^2_2$ can either lie entirely in $T_1$ or $T_2$; or precisely two leaves in each of the branches $B_1$ and $B_2$ of $T$ induce the star $S_2$ to obtain a copy of $CD^2_2$. This proves the recursive formula for $CD^2_2$.
\end{itemize}
Thus, it is never necessary to store full tree structures. At the end, the maximum 
\begin{align*}
\max_{\substack{|T| = n\\ T~\text{binary tree}}} c(A_5,T)
\end{align*}
can be determined easily from $L(n)$.

\section{Proof of Theorem~\ref{lowerbforA5}}\label{SectProofA5}

This section is devoted to proving Theorem~\ref{lowerbforA5}. Recall that we are going to use the binary tree $S(n_1,n_2,n_3,n_4)$ presented in Figure~\ref{Gene Tree for A2,5}. Moreover, we now need to consider only $I_2(A_5)$ because it is established in~\citep[Corollary 8]{DossouOloryWagnerPaper3} that $J(B)=I_2(B)$ for every binary tree $B$.

\begin{proof}[of Theorem~\ref{lowerbforA5}]
	Let us set \begin{math}n=n_1+n_2+n_3+n_4\end{math}. Recall from equation~\eqref{A5inT} that a recursion for the number of copies of $A_5$ in any binary tree $T$ with branches $B_1$ and $B_2$ is given by
	\begin{align*}
	c(A_5,T)=c(A_5,B_1)+c(A_5,B_2)+|B_1|\cdot c\big(CD^2_2,B_2\big) + |B_2|\cdot c\big(CD^2_2,B_1\big)\,.
	\end{align*}
	So for the tree $S(n_1,n_2,n_3,n_4)$, we obtain
	\begin{align}\label{First1}
	\begin{split}
	c\big(A_5,S(n_1,n_2,n_3,n_4)\big)&=c\big(A_5,E^2_{n_1}\big)+c\big(A_5,E^2_{n_2}\big) +c\big(A_5,E^2_{n_3}\big) + c\big(A_5,E^2_{n_4}\big)\\
	& +n_3 \cdot c\big(CD^2_2,E^2_{n_4}\big) + n_4 \cdot c\big(CD^2_2,E^2_{n_3}\big)\\
	&+n_2 \cdot c\big(CD^2_2,T_{n_3,n_4}\big) + (n_3+n_4) \cdot c\big(CD^2_2,E^2_{n_2}\big)\\
	&+n_1 \cdot c\big(CD^2_2,T_{n_2,n_3,n_4}\big) + (n_2+n_3+n_4) \cdot c\big(CD^2_2,E^2_{n_1} \big)\,,
	\end{split}
	\end{align}
	where $T_{n_3,n_4}$ is the binary tree whose branches are the even binary trees $E^2_{n_3}$ and $E^2_{n_4}$, while $T_{n_2,n_3,n_4}$ is the binary tree whose branches are $E^2_{n_2}$ and $T_{n_3,n_4}$.
	
	Also, recall from equation~\eqref{eqforCD2} that a recursion for the number of copies of $CD_2^2$ in any binary tree $T$ with branches $B_1$ and $B_2$ is given by
	\begin{align*}
	c(CD_2^2,T) = c(CD_2^2,B_1) + c(CD_2^2,B_2) + \binom{|B_1|}{2} \binom{|B_2|}{2}\,.
	\end{align*}
	So for the binary tree $T_{n_2,n_3,n_4}$, we get
	\begin{align*}
	c\big(CD^2_2,T_{n_2,n_3,n_4}\big)=c\big(CD^2_2,E^2_{n_2}\big)+c\big(CD^2_2,T_{n_3,n_4} \big)+\binom{n_2}{2} \binom{n_3+n_4}{2}\,.
	\end{align*}
	Likewise, 
	\begin{align*}
	c\big(CD^2_2,T_{n_3,n_4}\big)=c\big(CD^2_2,E^2_{n_3}\big)+c\big(CD^2_2,E^2_{n_4} \big) +\binom{n_3}{2} \binom{n_4}{2}\,.
	\end{align*}
	
	Thus, equation~\eqref{First1} becomes
	\begin{align*}
	c\big(A_5,\,&S(n_1,n_2,n_3,n_4)\big)=c\big(A_5,E^2_{n_1}\big)+c\big(A_5,E^2_{n_2}\big) +c\big(A_5,E^2_{n_3}\big) + c\big(A_5,E^2_{n_4}\big)\\
	& + n_3 \cdot c\big(CD^2_2,E^2_{n_4}\big) + n_4 \cdot c\big(CD^2_2,E^2_{n_3}\big)+ (n_3+n_4) \cdot c\big(CD^2_2,E^2_{n_2}\big)\\
	&+n_2  \Big(c\big(CD^2_2,E^2_{n_3}\big)+c\big(CD^2_2,E^2_{n_4} \big) +\binom{n_3}{2} \binom{n_4}{2} \Big) \\
	&+n_1 \Big( c\big(CD^2_2,E^2_{n_2}\big)+ c\big(CD^2_2,E^2_{n_3}\big)+c\big(CD^2_2,E^2_{n_4} \big) +\binom{n_3}{2} \binom{n_4}{2}\\
	&+\binom{n_2}{2} \binom{n_3+n_4}{2}\Big) + (n_2+n_3+n_4) \cdot c\big(CD^2_2,E^2_{n_1} \big)
	\end{align*}
	after combining everything. As a special case of \citet[Theorem 12]{DossouOloryWagnerPaper2}, we have
	\begin{align*}
	c\big(CD^2_2,E^2_n\big)=\frac{1}{56}\cdot n^4 +\mathcal{O}(n^3)
	\end{align*}
	for all $n$. On the other hand, using this asymptotic formula along with the recursion
	\begin{align*}
	c\big(A_5,E^2_n\big)&=c\big(A_5,E^2_{\lfloor n/2 \rfloor}\big)+c(A_5,E^2_{\lceil n/2 \rceil})\\
	&+\lfloor n/2 \rfloor \cdot c\big(CD^2_2,E^2_{\lceil n/2 \rceil}\big) + \lceil n/2 \rceil \cdot c\big(CD^2_2,E^2_{\lfloor n/2 \rfloor}\big)\,,
	\end{align*}
	which follows from~\eqref{A5inT} by the definition of the even binary tree $E^2_n$, it is not hard to prove that there exist absolute constants $K_1,K_2\geq 0$ such that the double inequality
	\begin{align*}
	\frac{1}{840}\cdot n^5 - K_1\cdot n^4\leq c\big(A_5,E^2_n\big)\leq \frac{1}{840}\cdot n^5 +K_2\cdot n^4
	\end{align*}
	holds for all $n$---the details are omitted. 
	
	Now, let $x_1,x_2,x_3,x_4$ be positive real numbers with \begin{math}x_1+x_2+x_3+x_4=1\end{math}. We set 
	\begin{align*}
	n_i = \lfloor x_i n \rfloor = x_i n + \mathcal{O}(1)
	\end{align*}
for $i=1,2,3,4$. Combining all the asymptotic formulas, we can now rewrite \begin{math}c\big(A_5,S(n_1,n_2,n_3,n_4)\big)\end{math} as follows:
	\begin{align*}
	c\big(A_5,S(n_1,n_2,n_3,n_4)\big)&=\frac{n^5}{840} \big(x_1^5+x_2^5 +x_3^5 +x_4^5\big)+\frac{n^5}{56} \big(x_3\cdot x_4^4 +x_4\cdot x_3^4\big)\\
	&+\frac{n^5}{56}\cdot x_2\big(x_3^4 +x_4^4  +14 \cdot x_3^2 \cdot x_4^2 \big) + \frac{n^5}{56}(x_3+x_4)x_2^4 \\
	&+\frac{n^5}{56}\cdot x_1 \Big(x_2^4+ x_3^4 + x_4^4+ 14 \cdot x_3^2\cdot x_4^2  + 14\cdot x_2^2 (x_3+x_4)^2\Big)\\
	& +\frac{n^5}{56}(x_2+x_3+x_4)x_1^4 +\mathcal{O}(n^4)\,.
	\end{align*}
	
	Set
	\begin{align*}
	F(x_1,x_2,x_3)&=\frac{1}{840} \big(x_1^5+x_2^5 +x_3^5 +(1-x_1-x_2-x_3)^5\big)+\frac{1}{56} \bigg(x_3 (1-x_1-x_2-x_3)^4\\
	& +(1-x_1-x_2-x_3) x_3^4 + x_2\big(x_3^4 +(1-x_1-x_2-x_3)^4 \\
	& +14 \cdot x_3^2 (1-x_1-x_2-x_3)^2 \big) + (1-x_1-x_2)x_2^4 \\
	&+x_1 \Big(x_2^4+ x_3^4 + (1-x_1-x_2-x_3)^4+ 14 \cdot x_3^2 (1-x_1-x_2-x_3)^2 \\
	& + 14\cdot x_2^2 (1-x_1-x_2)^2\Big) +(1-x_1)x_1^4 \bigg)\,.
	\end{align*}
	Then we obtain
	\begin{align*}
	c\big(A_5,S(n_1,n_2,n_3,n_4)\big)=F(x_1,x_2,x_3) \cdot n^5 +\mathcal{O}(n^4)
	\end{align*}
	as \begin{math}x_1+x_2+x_3+x_4=1\end{math}. With the help of a computer, we find that the global maximum of the function $F(x_1,x_2,x_3)$ in the region covered by the inequalities \begin{math}0<x_1,x_2,x_3<1,~x_1+x_2+x_3<1\end{math} is attained at the points whose values are approximately
	\begin{align*}
	(x_1=0.025347732268,x_2=0.051425755177,x_3=0.788023120078)
	\end{align*}
	and
	\begin{align*}
	(x_1=0.025347732268, x_2=0.051425755177, x_3=0.135203392478).
	\end{align*}
	Note here that the two critical points are related by the symmetry between $x_3$ and $x_4$ in the asymptotic formula for \begin{math}c\big(A_5,S(n_1,n_2,n_3,n_4)\big)\end{math}.
	Thus we have
	\begin{align*}
	\max_{\substack{0<x_1,x_2,x_3<1 \\ x_1+x_2+x_3<1}}F(x_1,x_2,x_3) &\geq F(0.025347732268,0.051425755177,0.788023120078)\\
	& \approx 0.002058929182\,.
	\end{align*}
	
	The inequality
	\begin{align*}
	I_2(A_5)\geq 0.002058929182 \times 5! \approx 0.247071501785
	\end{align*}
	follows. This concludes the proof of the first part of the theorem.
	
	\begin{remark}\label{rem:minipoly}
		The precise lower bound is an algebraic number of degree $16$; its minimal polynomial is given in the appendix.
	\end{remark}
	
	\medskip
	For the upper bound, we make use of Theorem~\ref{Thm:forAlgo} which states that the maximum of $c(A_5,T)$ for binary trees $T$ with $n$ leaves can be determined purely by focusing on the sets $\cL(n)$ whose algorithmic description is given in Section~\ref{sec:ALGO}. Recall that by Theorem~3 in~\citep{AudaceStephanPaper1}, we have
	\begin{align*}
	I_2(A_5) \leq \max_{\substack{|T|=n \\ T~\text{binary tree}}}\gamma(A_5,T)
	\end{align*}
	for every $n \geq 5$. Thus we want to calculate the maximum for different values of $n$. When our algorithm terminates, the maximum number of copies of $A_5$ among all binary trees with $n$ leaves can be read off as the greatest $x$-coordinate (first coordinate) of the elements of $L(n)$, that is the $x$-coordinate of the very first element of $L(n)$---see the discussion before Theorem~\ref{Thm:forAlgo}.
	
	\medskip
	We have implemented this algorithm in Mathematica. The notebook can be accessed at \url{http://math.sun.ac.za/~swagner/TreeA5Final}. The precise values of
	\begin{align*}
	a_n =\max_{\substack{|T|=n \\ T~\text{binary tree}}}\gamma(A_5,T)
	\end{align*}
	have been computed for $n\leq 2000$---see Table~\ref{tab:treeA5}. It follows that 
	\begin{align*}
	I_2(A_5) \leq a_{2000} = \frac{32828685715097}{132667832500200}\approx 0.24745\,.
	\end{align*}
	This completes the proof of the theorem.

\begin{table}[h]\centering
	\caption {Maximum density $a_n$ of $A_5$ among $n$-leaf binary trees}\label{tab:treeA5} 
	\begin{tabular}{c | c |c |c |c |c |c |c |c }
		\hline
		$\boldsymbol{n}$ & 5 & 6 & 7 & 8 & 9 & 10 & 20 & 30 \\ \hline
		$a_n$ & $1$ & $\frac{1}{2}$ & $\frac{3}{7}$ & $\frac{11}{28}$ & $\frac{23}{63}$ & $\frac{1}{3}$ & $\frac{553}{1938}$ &  $\frac{19219}{71253}$ \Tstrut\Bstrut \\ \hline \hline
		$\boldsymbol{n}$ & 40 & 50 & 60 & 70 & 80 & 90 & 100 & 150  \\ \hline
		$a_n$ & $\frac{57793}{219336}$ & $\frac{550621}{2118760}$ &  $\frac{351943}{1365378}$ & $\frac{44899}{175406}$ & $\frac{6127045}{24040016}$ & $\frac{930032}{3662439}$ & $\frac{3177631}{12547920}$ & $\frac{24765738}{98600005}$  \Tstrut\Bstrut \\ \hline \hline
		$\boldsymbol{n}$ & 200 & 250 & 300 & 350 & 400 & 500 & 600 & 700  \\ \hline
		$a_n$ & $0.250153$ & $0.249543$ & $0.249142$ & $0.248854$ & $0.24864$ & $0.24834$ & $0.248143$ & $0.248001$ \Tstrut\Bstrut \\ \hline \hline
		$\boldsymbol{n}$ & 800 & 900 & 1000 & 1200 & 1400 & 1600 & 1800 & 2000 \\ \hline
		$a_n$ & $0.247894$ & $0.247812$ & $0.247747$ & $0.247648$ & $0.247577$ & $0.247524$ & $0.247483$ & $0.24745$ \Tstrut\Bstrut \\ \hline
	\end{tabular}
\end{table}
\end{proof}

\section{Proof of Theorem~\ref{accurate lower bound for P4}}\label{Proof2}

Let us first provide a proof of Proposition~\ref{prop: copy of P4}, which is an intermediate step in the proof of Theorem~\ref{accurate lower bound for P4}:

\begin{proof}[of Proposition~\ref{prop: copy of P4}]
	Let $k\geq 3$ and $d\geq k-1$ be fixed (note that $c\big(Q_k,CD^d_h \big)=0$ for $d<k-1$). For $h=1$, we have $c\big(Q_k,CD^d_h \big)=c(Q_k,S_d)=0$. Since for the case $h=1$, the statement holds trivially, we can safely assume $h\geq 2$ and proceed by induction on $h$.
	
	We distinguish possible cases that can happen for a subset of $k$ leaves of the tree $CD^d_h$: 
	\begin{itemize}
		\item all $k$ leaves belong to the same branch of $CD^d_h$. The total number of these subsets of leaves that induce the tree $Q_k$ is given by \begin{math}d \cdot c\big(Q_k,CD^d_{h-1} \big)\end{math}, as all the branches of $CD^d_h$ are isomorphic to $CD^d_{h-1}$;
		\item more than two of the branches of $CD^d_h$ contain at least one of the $k$ leaves. In this case the leaf-induced subtree is not isomorphic to $Q_k$ (as the root degree of $Q_k$ is $2$);
		\item exactly two of the branches of $CD^d_h$ contain at least two of the $k$ leaves each. In this case the leaf-induced subtree is not isomorphic to $Q_k$ (as one of the branches of $Q_k$ is the single leaf);
		\item one branch of $CD^d_h$ contains exactly one of the leaves and another branch of $CD^d_h$ contains $k-1$ leaves. Since $k>2$, the total number of these subsets of leaves that induce the tree $Q_k$ is given by \begin{math}2\cdot d^{h-1} \cdot c\big(S_{k-1},CD^d_{h-1}\big)\end{math} for every choice of two branches of $CD^d_h$.
	\end{itemize}
	Therefore, a recursion for $c\big(Q_k,CD^d_h\big)$ is given by
	\begin{align*}
	c\big(Q_k,CD^d_h \big)&=d \cdot c\big(Q_k,CD^d_{h-1} \big)+2 \dbinom{d}{2} \cdot d^{h-1} \cdot c\big(S_{k-1},CD^d_{h-1}\big)\\
	&=d \cdot c\big (Q_k,CD^d_{h-1} \big)+(d-1) d^h \binom{d}{k-1} \Bigg(\frac{d^{(k-1)(h-1)}-d^{h-1}}{d^{k-1}-d} \Bigg)\,,
	\end{align*}
	where the last step uses the identity
	\begin{align}\label{copyCkCDh}
	c\big(S_k,CD^d_h\big)= \binom{d}{k} \frac{d^{k\cdot h}-d^h}{d^k-d}
	\end{align}
	valid for every $k\geq 2$ -- formula~\eqref{copyCkCDh} can be found in~\citep[proof of Theorem~1]{AudaceStephanPaper1}. The induction hypothesis gives
	\begin{align*}
	c\big(Q_k,CD^d_h \big)=~&\frac{(d-1) \binom{d}{k-1}}{d^{k-1}-d}\cdot d^h \Bigg(\frac{d^{(k-1)(h-1)}-d^{k-1}}{d^{k-1}-1}-\frac{d^{h-1}-d}{d-1} \Bigg)\\
	&~~~~ +(d-1) d^h \binom{d}{k-1} \Bigg(\frac{d^{(k-1)\cdot (h-1)} -d^{h-1}}{d^{k-1}-d} \Bigg)\\
	=~&\frac{(d-1) \binom{d}{k-1}}{d^{k-1}-d}\cdot d^h \Bigg( d^{(k-1)(h-1)} +\frac{d^{(k-1)(h-1)}-d^{k-1}}{d^{k-1}-1}\\
	& \hspace*{5cm} -\frac{d^{h-1}-d}{d-1} -d^{h-1}\Bigg)\,,
	\end{align*}
	which, after simplification, yields the desired equality. The statement on the inducibility follows by passing to the limit of the density $\gamma\big(Q_k,CD^d_h\big)$ as $h \to \infty$:
	\begin{align*}
	I_d(Q_k)\geq \lim_{h \to \infty}\gamma\big(Q_k,CD^d_h\big)=\frac{k!(d-1) \binom{d}{k-1}}{(d^{k-1}-d) (d^{k-1}-1)}\,.
	\end{align*}
\end{proof}

We can now focus on Theorem~\ref{accurate lower bound for P4}.

\begin{proof}[of Theorem~\ref{accurate lower bound for P4}] 
	First off, we construct a new family of ternary trees: given a nonnegative integer $h\geq 0$, attach one copy of each of the complete ternary trees $CD^3_h$ and $CD^3_{h+1}$ to a common vertex (their respective roots are joined to a new vertex) to form a ternary tree which we shall name $W^3_{h}$. For example, $W^3_0$ is the tree $Q_4$. See also Figure~\ref{TreeW3h} for the ternary tree $W^3_1$.
	\begin{figure}[htbp]\centering  
		\begin{tikzpicture}[thick,level distance=11.5mm]
		\tikzstyle{level 1}=[sibling distance=35mm]
		\tikzstyle{level 2}=[sibling distance=12mm]
		\tikzstyle{level 3}=[sibling distance=4mm]
		\node [circle,draw]{}
		child {child {[fill] circle (2pt)}child {[fill] circle (2pt)}child {[fill] circle (2pt)}}
		child {child {child {[fill] circle (2pt)}child {[fill] circle (2pt)}child {[fill] circle (2pt)}}child {child {[fill] circle (2pt)}child {[fill] circle (2pt)}child {[fill] circle (2pt)}}child {child {[fill] circle (2pt)}child {[fill] circle (2pt)}child {[fill] circle (2pt)}}};
		\end{tikzpicture}
		\caption{The ternary tree $W^3_1$ defined in the proof of Theorem~\ref{accurate lower bound for P4}.}\label{TreeW3h}
	\end{figure}
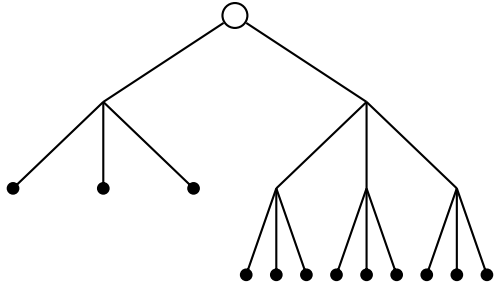
	
	We shall prove that 
	\begin{align*}
	\lim_{h\to \infty} \gamma \big(Q_4,W^3_h \big)=\frac{59}{416}\,.
	\end{align*}
	
	To justify the specific choice, let us consider a more general construction. For positive integers $n_1$ and $n_2$, we consider the ternary tree (which we simply denote by $T_{n_1,n_2}$) whose branches are even ternary trees with $n_1$ and $n_2$ leaves, respectively. The even ternary tree $E^3_n$ with $n$ leaves is obtained recursively as follows: $E^3_1$ is the tree with only one vertex; $E^3_2$ is the star with two leaves; for $n>2$, the branches of $E^3_n$ are the even ternary trees $E^3_{k_1},~E^2_{k_2}$ and $E^3_{k_3}$ with $k_1,k_2,k_3$ as equal as possible and $k_1+k_2+k_3=n$.
	
	\medskip
	According to Proposition~\ref{prop: copy of P4}, we have
	\begin{align*}
	c\big(Q_4,CD^d_h\big)=\frac{(d-2)d^{4\cdot h}}{6(d+1)(d^2+d+1)}+\mathcal{O}(d^{2\cdot h})
	\end{align*} 
	for every $d$ and all $h\geq 3$. In particular, the asymptotic formula
	\begin{align*}
	c\big(Q_4,CD^3_h\big)=\frac{1}{312}\cdot 3^{4\cdot h}+\mathcal{O}(3^{2\cdot h})
	\end{align*}
	is obtained for all $h\geq 3$. On the other hand, we recall that the specialisation $k=3$ in equation~\eqref{copyCkCDh} of the proof of Proposition~\ref{prop: copy of P4} gives
	\begin{align*}
	c(S_3,CD^3_h)=\frac{1}{24}\cdot 3^{3h}+\mathcal{O}(3^h)
	\end{align*}
	for all $h\geq 1$, and employing the identity
	\begin{align*}
	c(S_3,E^3_n)=c(S_3,E^3_{k_1})+c(S_3,E^3_{k_2})+c(S_3,E^3_{k_3})+k_1\cdot k_2 \cdot k_3\,,
	\end{align*}
	it is not difficult to show that 
	\begin{align*}
	c(S_3,E^3_n)=\frac{1}{24}\cdot n^3+\mathcal{O}(n^2)
	\end{align*}
	for all $n$. Using the recursion
	\begin{align*}
	c\big(Q_4,E^3_n\big)&=c\big(Q_4,E^3_{k_1}\big)+c\big(Q_4,E^3_{k_2}\big)+c\big(Q_4,E^3_{k_3}\big)+k_1 \cdot c\big(S_3,E^3_{k_2}\big)+k_2 \cdot c\big(S_3,E^3_{k_1}\big)\\
	&+k_1 \cdot c\big(S_3,E^3_{k_3}\big) +k_3 \cdot c\big(S_3,E^3_{k_1}\big)+k_2 \cdot c\big(S_3,E^3_{k_3}\big)+k_3 \cdot c\big(S_3,E^3_{k_2}\big)\,,
	\end{align*}
	we also find that
	\begin{align*}
	c\big(Q_4,E^3_n\big)=\frac{1}{312}\cdot n^4 +\mathcal{O}(n^3)
	\end{align*}
	for all $n$. Moreover, the number of copies of $Q_4$ in any topological tree $T$ with two branches $T_1,T_2$ is given by
	\begin{align*}
	c(Q_4,T)=c(Q_4,T_1)+c(Q_4,T_2)+|T_1|\cdot c(S_3,T_2)+|T_2|\cdot c(S_3,T_1)\,.
	\end{align*}
	
	For $x \in (0,1)$, set $n_1 = \lfloor x\cdot n \rfloor$ and $n_2 = \lfloor (1-x)n \rfloor$, and let $n \to \infty$. Combining all the formulas above, we see that an asymptotic formula for $c\big(Q_4,T_{n_1,n_2}\big)$ is given by
	\begin{align*}
	c\big(Q_4,T_{n_1,n_2}\big)&=\frac{1}{312}(x\cdot n)^4 +\frac{1}{312}\big((1-x) n\big)^4\\
	&~~ +x\cdot n \cdot \frac{1}{24}\big((1-x) n\big)^3 +(1-x) n \cdot \frac{1}{24} (x\cdot n)^3  + \mathcal{O}(n^3)\\
	&=\frac{1}{312} \big(x^4+(1-x)^4\big) n^4 +\frac{1}{24} \big(x (1-x)^3+(1-x) x^3\big)n^4 +\mathcal{O}(n^3)\\
	&=\frac{1}{312} \big( 1 + 9x - 33x^2 + 48x^3 - 24x^4 \big)n^4 + \mathcal{O}(n^3)\,.
	\end{align*}
	
	Set
	\begin{align*}
	f(x)=\frac{1}{312} \big( 1 + 9x - 33x^2 + 48x^3 - 24x^4 \big)\,.
	\end{align*}
	The first derivative of this function is given by
	\begin{align*}
	f^{\prime}(x)=\frac{-(2x-1)(4x-3)(4x-1)}{104}\,.
	\end{align*}
	We see that $f(x)$ attains its maximum at $x= 1/4$ (or $x=3/4$):
	\begin{align*}
	f(x)\leq f\Big(\frac{1}{4}\Big)=\frac{59}{9984}
	\end{align*}
	for all $x \in (0,1)$. This motivates the choice of the trees $W^3_h$ defined before. We have
	\begin{align*}
	I_3(Q_4)\geq \lim_{h\to \infty} \gamma\big(Q_4,W^3_h\big)= \frac{4!\cdot 59}{9984}=\frac{59}{416}\,.
	\end{align*}
	This completes the proof of the lower bound in the theorem.
	
	\medskip
	The proof of the upper bound is also via an algorithmic approach and is quite similar to the one given for the binary tree $A_5$ in Section~\ref{sec:ALGO}. 
	Recall again that by Theorem~3 in~\citep{AudaceStephanPaper1}, 
	\begin{align*}
	I_3(Q_4) \leq \max_{\substack{|T|=n \\ T~\text{ternary tree}}}\gamma(Q_4,T),
	\end{align*}
	so the aim is to compute the right hand side for different values of $n$. The algorithm is essentially the same as for $A_5$, with the trees $Q_4$ and $S_3$ assuming the roles of $A_5$ and $CD_2^2$ respectively. The only difference is that trees with two or three branches have to be considered in the construction of the sets $\cL(n)$.
	
	For the recursive calculation of $c(Q_4,T)$ and $c(S_3,T)$, we have the formulas
	\begin{align*}
	c(Q_4,T)&=c(Q_4,T_1)+c(Q_4,T_2)+c(Q_4,T_3)+|T_1|\cdot c(S_3,T_2)+|T_2|\cdot c(S_3,T_1)\\
	&+|T_1|\cdot c(S_3,T_3)+|T_3|\cdot c(S_3,T_1)+|T_2|\cdot c(S_3,T_3)+|T_3|\cdot c(S_3,T_2)
	\end{align*}
	and
	\begin{align*}
	c(S_3,T)=c(S_3,T_1)+c(S_3,T_2)+c(S_3,T_3)+|T_1|\cdot |T_2|\cdot |T_3|\,,
	\end{align*}
	where $T_1,T_2,T_3$ are the branches of $T$. If there are only two branches, all terms involving $T_3$ can simply be left out.
	
	\medskip
	Again, we have implemented the algorithm in Mathematica---the notebook can be found at \url{http://math.sun.ac.za/~swagner/TreeQ4Final}. The exact values of
	\begin{align*}
	b_n =\max_{\substack{|T|=n \\ T~\text{ternary tree}}}\gamma(Q_4,T)
	\end{align*}
	have been determined for values of $n$ up to $500$; see Table~\ref{tab:treeP4}. 
	
	\begin{table}[h]\centering
		\caption {Maximum density $b_n$ of $Q_4$ among $n$-leaf ternary trees}\label{tab:treeP4} 
		\begin{tabular}{c | c |c |c |c |c |c |c |c |c |c }
			\hline
			$n$ & 4 & 5 & 6 & 7 & 8 & 9 & 10  \\ \hline
			$\boldsymbol{b_n}$ & $1$ & $\frac{2}{5}$ & $\frac{2}{5}$ & $\frac{2}{7}$ & $\frac{19}{70}$ & $\frac{5}{21}$ & $\frac{5}{21}$ \Tstrut\Bstrut \\ \hline \hline
			$n$ & 15 & 20 & 25 & 30 & 35 & 40 & 45 \\ \hline
			$\boldsymbol{b_n}$ & $\frac{18}{91}$ & $\frac{291}{1615}$ & $\frac{1103}{6325}$ & $\frac{172}{1015}$ & $\frac{1097}{6545}$ & $\frac{7452}{45695}$ & $\frac{7948}{49665}$ \Tstrut\Bstrut \\ \hline \hline
			$n$ & 50 & 60 & 70 & 80 & 90 & 100 & 150 \\ \hline 
			$\boldsymbol{b_n}$ & $0.158072$ & $0.155422$ &  $0.153588$ & $0.152096$ & $0.150978$ & $0.150264$ & $0.147342$ \Tstrut\Bstrut \\ \hline \hline
			$n$ & 200 & 250 & 300 & 350 & 400 & 450 & 500 \\ \hline 
			$\boldsymbol{b_n}$ & $0.145967$ & $0.145195$ &  $0.144651$ & $0.144239$ & $0.143931$ & $0.143691$ & $0.143506$  \Tstrut\Bstrut \\ \hline \hline
		\end{tabular}
	\end{table}
	
	We conclude that      			 		
	\begin{align*}
	I_3(Q_4) \leq b_{500}= \frac{73848853}{514606225} \approx 0.143506\,.
	\end{align*}
	This completes the proof of the theorem.
\end{proof}

\section*{Conclusion}

Naturally, the main open question left for us is

\begin{question}
	What are the precise values of $J(A_5) = I_2(A_5)$ and $I_3(Q_4)$?
\end{question}

It is conceivable that the constructions yielding our lower bounds are asymptotically optimal, in which case the lower bounds would in fact be exact. A characterisation (at least an approximate characterisation) of the trees that attain the maxima of $\gamma(A_5,T)$ and $\gamma(Q_4,T)$ would be highly desirable.

\medskip

Observe that the only known exact values of the inducibility are rational numbers. It was already asked in \citep{czabarka2016inducibility} whether the inducibility of trees (binary trees in that specific paper) is always rational. A natural generalisation would be the following question:

\begin{question}
	Are $I_d(S)$ and $J(S)$ rational for every rooted tree $S$ (and every integer $d \geq 2$)?
\end{question}

The tree $A_5$ that we studied in this paper would be a natural candidate for a counterexample. For instance, our bounds show that $J(A_5) = I_2(A_5)$, if rational, would have to have a denominator of at least $89$. Thus an answer to the first question might immediately imply an answer to our second question.

\medskip
Another natural direction of further research would be to search for other examples of trees whose inducibility can be determined explicitly (especially if this turns out to be too difficult for the trees $A_5$ and $Q_4$).

\medskip
Finally, we would like to mention that there are other small binary trees $B$ for which the same algorithm described for the tree $A_5$ (see Section~\ref{sec:ALGO}) can be exploited to determine the maximum number of copies of $B$ in a binary tree with $n$ leaves, thus an upper bound on $I_2(B)=J(B)$. Let $B_1$ and $B_2$ denote the two branches of $B$, and $B_{2,1},B_{2,2}$ the two branches of $B_2$. If \begin{math}1 \leq |B_1|,|B_{2,1}|,|B_{2,2}|\leq 3\end{math}, then one can apply the same algorithm with only very minor modifications. There are $17$ nonisomorphic binary trees satisfying this criterion. The inducibility is presently known explicitly for only $7$ of these trees.

\section*{Appendix}

The minimal polynomial of the lower bound derived in Section~\ref{SectProofA5} (cf. Remark~\ref{rem:minipoly}) is

\begin{align*}
&-219990282547586266429960528777627452703544325176405813669341 \\
&+  14602043726049732276047519980572925148798805701655918709812280 x \\
&-  443988064886113118898743858593495837271116452775244720945246560 x^2 \\
&+ 8191391786597997025923387108156673725457502710845806254299710400 x^3 \\
&- 102349758416566196856322057341155143983386721744416045107152357780 x^4 \\
&+ 914915733104054427549320025907848536757859198555663276682925151464 x^5 \\
&- 6021742541574757636997532900244251351617306701953896563191732661600 x^6 \\
&+ 29555503633329799978352177635679651562063414854261341767132451211440 x^7 \\
&- 108203641960712037979399490009159473059103361912934081445369533569710 x^8 \\
&+ 291888020622671692818879080374814508443375281580239785514869152091240 x^9 \\
&- 563951453122800910206893287609142225349017486796615704590365725141664 x^{10} \\
&+ 738771836341212349165479496191602729493266527262931365590196142374880 x^{11} \\
&- 587213314414708394727507148742441667728136148596591603986693048673940 x^{12} \\
&+ 211982553160494718288945301048132425769769562499489523070365996462520 x^{13} \\
&+ 2642044260670867601071997587023204415009550503501324422149424398240 x^{14} \\
&+ 2298958777082465800903908865165860713406872658260092533772950000 x^{15} \\
&+ 563916767637963643123242260073274239273437824576171875 x^{16}.
\end{align*}

\end{document}